\documentclass[10pt]{amsart}

\usepackage[top=3.5cm,bottom=3.5cm,left=3.7cm,right=3.7cm]{geometry}               
\usepackage{amsmath,amscd,amssymb,amsthm}
\usepackage[english]{babel}
\usepackage{mathtools}
\usepackage{enumerate}
\usepackage{setspace}
\usepackage{mathrsfs}
\usepackage[numbers]{natbib}
\usepackage[hidelinks]{hyperref}
\usepackage{csquotes}
\usepackage{tkz-base}
\usepackage{colonequals}

\author{{\textbf{Giacomo Micheli}} \\ with an appendix by Mihran Papikian} 

\title{\Large{On the Riemann Hypothesis for Drinfeld Modules}}

\usepackage{fancyhdr}
\DeclareMathOperator{\End}{End}

\def\fp{\mathfrak{p}}
\def\fl{\mathfrak{l}}
\def\fq{\mathfrak{q}}

\def\vF{\mathbb{F}}
\def\vN{\mathbb{N}}
\def\vZ{\mathbb{Z}}

\def\vR{\mathbb{R}}

\urldef\myurlcrm\url{ http://www.crm.cat/en/activities/pages/activityfoldersandpages/curs}

\def\cO{\mathcal{O}}

\def\ff{\mathfrak{f}}
\def\fm{\mathfrak{m}}
\def\fw{\mathfrak{w}}
\newcommand{\tinf}{\widetilde{\infty}}
\newcommand{\ls}[2]{#1\!\left(\mkern-4mu\left(#2\right)\mkern-4mu\right)}
\newcommand{\Mod}[1]{\ (\mathrm{mod}\ #1)}
\DeclareMathOperator{\Nr}{N}

\newtheorem{theorem}{Theorem}

\newtheorem{lemma}[theorem]{Lemma}
\newtheorem{corollary}[theorem]{Corollary}

\newtheorem{proposition}[theorem]{Proposition}
\theoremstyle{definition}
\newtheorem{definition}[theorem]{Definition}

\newtheorem{remark}[theorem]{Remark}

\makeindex

\date{}

\begin{document}

\begin{abstract}
In this paper we provide a short proof of the Riemann Hypothesis for Drinfeld modules which uses only standard notions from the theory of global function fields and of Drinfeld modules.
\end{abstract}

\maketitle

\section{Introduction}

The Riemann Hypothesis for Drinfeld modules (RHfDM) is a classical result in algebraic number theory and has been proved in various ways: see for example, \cite{drinfel1977elliptic,laumon1996cohomology,Papikian2023,yu1995isogenies}  to name some. One of the most accessible is the one in \cite[Chapter 4]{Papikian2023}, which nevertheless uses some deep facts on the theory of cyclic algebras (see for example \cite[Proposition 4.1.1, Theorem 4.1.3]{Papikian2023}).
By contrast, aside from standard algebra facts, the statement of RHfDM only requires quite basic knowledge of function fields, and Drinfeld module theory (essentially only the definition of Tate module), see Theorem \ref{thm:RH}. The purpose of this paper is to produce a short proof that only requires basic facts from the theory of extensions of global fields and basic facts on Drinfeld modules. The independence-of-$\fl$ result usually supplied by \cite[Theorem 3.6.6]{Papikian2023} is proved in Appendix~\ref{app:ell}; the only remaining imported structural fact is \cite[Corollary 3.4.11]{Papikian2023}.

Set $k=\vF_{q^n}$.
Let $k\{\tau\}$ be the skew polynomial ring in $k$ defined by the rule $\tau \alpha=\alpha^q\tau$, for any $\alpha\in k$.
Let $A=\vF_q[T]$ and $\fl$ be a prime ideal of $A$. With abuse of notation, we will indicate with $\fl$ both $\fl$ and its unique monic generator. 
Let $\phi:\vF_q[T]\rightarrow k\{\tau\}$ be a Drinfeld module over a finite field $k$, $\phi_{f(T)}$ be the image of $f(T)$, and $\End(\phi)$ be the ring of endomorphisms of $\phi$ (i.e.\ the ring consisting of the elements of $k\{\tau\}$ that commute with $\phi_T$). 
Let $P_{\phi,\pi}(T,x)$ be the characteristic polynomial of the Frobenius $\pi=\tau^n\in \End(\phi)$ as an  $A_\fl$-linear map acting on the Tate module $T_{\fl}(\phi)$. It is a well known fact that $P_{\phi,\pi}(T,x)$ belongs to $A[x]$ and is independent of $\fl$; a proof is given in Appendix~\ref{app:ell}. If $a\in A$, let us denote by $(a)$ the ideal generated by $a$.
If $\phi$ is a Drinfeld module, let $\phi_T=\sum^r_{i=0}g_i \tau^i$.  The characteristic of $\phi$ is defined to be the minimal polynomial over $\vF_q$ of $g_0$.
Let $L$ be a global field, i.e. a finite extension of $\vF_q(T)$, and let $K$ be a subfield of $L$. Then we define the norm function as 
\[N_{L:K}:L\rightarrow K\]
\[x\mapsto \det(M_x)\]
where $M_x$ is the $K$-linear map on $L$ defined by the multiplication by $x$. Whenever the extension is understood, we will simply write $N$ instead of $N_{L:K}$.
If $\infty$ is the place at infinity of $\vF_q(T)$ let $v_\infty$ be the normalized valuation on $\vF_q(T)$ defined by $v_{\infty}(f/g)=\deg(g)-\deg(f)$.

\begin{theorem}[Riemann Hypothesis for Drinfeld Modules]\label{thm:RH}
Let $\vF_q$ be a finite field and $k$ be a finite extension of $\vF_q$.
Set $A=\vF_q[T]$.
Let 
\[\phi:A\longrightarrow k\{\tau\}\]
\[T\mapsto \sum^r_{i=0}g_i \tau^i\]
 be a Drinfeld module of rank $r$. Let $n=[k:\vF_q]$ and $d=\deg(\fp)$, where $\fp$ is the characteristic of $\phi$. 
Let
\[P_{\phi,\pi}(T,x)=x^r+\sum^{r-1}_{i=0}a_i x^i\in A[x]\]
be the characteristic polynomial of the Frobenius on the Tate module $T_{\fl}(\phi)$ with $\fl\neq \fp$.
Let $M$ be the splitting field of $P_{\phi,\pi}$. Let $\alpha\in M$ be a root of $P_{\phi,\pi}$, and $L=\vF_q(\alpha,T)$.
Consider the map
\begin{align*}
|\cdot|_{*}:L&\longrightarrow \mathbb R^{+}\\
u &\mapsto  |u|_{*}:=q^{-v_{\infty}(N_{\vF_q(u,T):\vF_q(T)}(u))/[\vF_q(u,T):\vF_q(T)]}.
\end{align*} 
Then, $|\cdot|_*$ is an absolute value. Furthermore, we have that
\begin{enumerate}
\item for any root $\alpha$ of $P_{\phi,\pi}(T,x)$ in $M$ we have that $|\alpha|_*=q^{n/r}$.
\item for any $i\in \{0,\dots, r-1\}$ we have that 
\[\deg(a_i)\leq (r-i)n/r\]
\item  $(a_0)=(\mathfrak p^{n/d})$.
\end{enumerate}
\end{theorem}
\begin{remark}
The core of the proof is establishing that $|\cdot|_{*}$ is an absolute value, then items (1) and (2) are quite standard if one knows the proof of Hasse-Weil Theorem for curves over finite fields, and uses basic facts on symmetric polynomials. We also follow these lines for such proofs when we conclude in Section \ref{sec:proofRH}.
\end{remark}
\begin{remark}
For simplicity, we compute $(a_0)$ and not $a_0$. The leading coefficient of $a_0$ can be obtained using for example the Weil pairing \cite[Theorem 3.7.1]{Papikian2023}.
\end{remark}

Fix a prime $\fq$ of $A$ different from the characteristic $\fp$ of
$\phi$. For $u\in\End(\phi)$, first define the local determinant
\[
\det {}_\fq(u):=\det\bigl(u\mid T_\fq(\phi)\bigr)
=(-1)^rP_{\phi,\fq,u}(T,0)\in A_\fq.
\]
Theorem~\ref{thm:charpol} shows that $\det_\fq(u)$ belongs to $A$ and is
independent of the choice of $\fq$. We denote this common element by
$\det(u)$; thus $\det:\End(\phi)\rightarrow A$ is well defined.

The core of our method is based on the combination of the following arguments. Let $u\in \End(\phi)$.
The first one consists of a decomposition of the kernel of a separable endomorphism whose determinant is not divisible by $\mathfrak p$ as an $\vF_q[T]$-module, showing that the degree of the determinant of $u$ on the Tate module is exactly the degree in $\tau$ of $u$ (Proposition \ref{prop:l-primary} and Corollary \ref{cor:valuationsep}). This allows to prove directly,  that $-\deg\det$ induces a valuation on $\vF_q(\phi_T,u)$ (Theorem \ref{thm:degdetval}), effectively extending to inseparable endomorphisms or to endomorphisms that have determinant divisible by $\mathfrak p$, using a perturbation argument. 
Then, by the independence of valuations, this forces that there is only one place above $\infty\subseteq \vF_q[T]$ (Lemmas \ref{lemma:techsplit} and \ref{lemma:uniqueplaceourcase}).
Section \ref{sec:proofRH} puts everything together \`a la Hasse-Weil and proves the special form of the characteristic polynomial of the Frobenius, including the computation of the absolute values of its roots.

Notice also that our proof does not even need that $\End(\phi)$ has $\vF_q[T]$ rank less than or equal to $r^2$, but only that it has finite rank, which we prove directly for completeness.

\section{Background}

\subsection{Modules over PIDs}

We recall here some basic facts from commutative algebra.

\begin{theorem}[Fundamental Theorem for Modules over a PID]\label{thm:ftmp}
Let $R$ be a PID and $M$ be a finitely generated $R$-module.
Then, there exist unique non-negative integers $t,\ell\in \vN\cup\{0\}$, and  unique (up to units) $a_1,\dots a_\ell\in R$ such that
\[M\cong R^{\oplus t}\oplus R/(a_1)\oplus\cdots\oplus R/(a_\ell)\] 
and $a_i\mid a_{i+1}$ for all $i\in \{1,\dots, \ell-1\}$.
\end{theorem}

\begin{definition}\label{def:orderideal}
In the notation of Theorem \ref{thm:ftmp}, and if $t=0$, we call $\chi_R(M):=(\prod^{\ell}_{i=1} a_i)$ the \emph{order ideal} of $M$.
\end{definition}

The following result is classical for modules over PIDs, see for example \cite[Theorem 1.2.6]{Papikian2023}.
\begin{theorem}\label{thm:chidet}
Let $R$ be a PID and $M$ be a free $R$-module of rank $n$. Let $u$ be an endomorphism of $M$ such that $u(M)$ is also a free $R$-module of rank $n$. Then,
\[\chi(M/u(M))=(\det(u))\]
\end{theorem}

\subsection{Function Fields}

We list here some useful definitions and results from the theory of valuations, absolute values, and of global function fields. Let $K$ be a global function field.
\begin{definition}
Two valuations $v,w$ on $K$ are said to be \emph{equivalent} if there exists a $c\in \vR_{> 0}$ such that $v(x)=cw(x)$ for all $x\in K$.
\end{definition}

\begin{definition}
Two non-archimedean absolute values  $|\cdot|_a,|\cdot|_b$ on $K$ are said to be \emph{equivalent} if there exists a $c\in \vR_{> 0}$ such that $|x|_a=|x|_b^c$ for all $x\in K$.
\end{definition}

If $Q$ is a place of a function field $F$, let  $v_Q$ be the normalized valuation attached to $Q$.

\begin{theorem}
Let $F/\vF_q$ be a function field. Then, every valuation on $F$ satisfies
\[v(x)=cv_Q(x)\]
for some place $Q$ of $F$ and some positive $c\in \vR^+$.
\end{theorem}

\begin{theorem}\label{thm:integrality}
Let $F=\vF_q(T)$ and $\overline F$ be the algebraic closure of $F$. Let $u$ be an element of $\overline F$ that is integral over $\vF_q[T]$. Then, the minimal polynomial of $u$ has coefficients in $\vF_q[T]$.
\end{theorem}
\begin{proof} This is a corollary of \cite[Theorem 3.3.1]{Stichtenoth} with $F'=\vF_q(u,T)$ and $z=u$.
\end{proof}

\begin{proposition}[Very Weak Independence of Valuations]
Let $M/\vF_q$ be a function field and let $P_1,\dots,P_\ell$ be places of $M$, and $y_1,\dots,y_\ell\in \vZ$. Then there exists an element $y\in M$ such that $v_{P_i}(y)=y_i$.
\end{proposition}
\begin{proof}
This is an immediate corollary of \cite[Theorem 1.3.1]{Stichtenoth} with $\ell=n$ (matching Stichtenoth's notation), the $x_i$'s equal to zero and $r_i=y_i$,  for $i\in \{1,\dots, \ell\}$.\end{proof}

\begin{lemma}\label{lemma:techsplit}
Let $M$ be an extension of $\vF_q(T)$. If the map $v:y\mapsto v_{\infty}(N(y))$ is a valuation on $M$ then  there is only one place of $M$ lying above $\infty\subseteq \vF_q(T)$.
\end{lemma}
\begin{proof}
Set $N=N_{M:\vF_q(T)}$.
Suppose that there are at least two places above $\infty$, say that $\tilde \infty$ is the place attached to $v$, and $Q$ is another place above $\infty$, with attached valuation $v_Q$. Say that $v_{\tilde \infty}$ is the normalized version of $v$. If $e$ is the ramification index of $Q$ over $\infty$, then $v_Q|_{\vF_q(T)}=e v_\infty$. This gives that there exists a positive coefficient $c$ such that for every $y\in M$ 
\begin{equation}\label{eq:contrwithval}
cv_{\tilde\infty}(y)=v(y)=v_{\infty}(N(y))=v_Q(N(y))/e
\end{equation}
Now, by independence of valuations fix $y\in M$ such that $v_{R}(y)=1$ for all $R\neq \tilde \infty$ above $\infty$ (in particular, $v_{Q}(y)=1$), and $v_{\tilde\infty}(y)=0$.

Now, $y\in \bigcap_{ R'\supseteq \infty} \mathcal O_{R'}$, which is the integral closure of $\mathcal O_\infty$ in $M$. 
Set $m=[\vF_q(y,T):\vF_q(T)]$. Consider now the minimal polynomial for
$y$, which lives in $\mathcal O_\infty[x]$:
\[x^m-\sum^{m-1}_{i=0}a_ix^i\]
and write 

\[y\left(y^{m-1}-\sum^{m-1}_{i=1}a_iy^{i-1} \right)=a_0.\]

The second factor on the left is integral at $Q$, so applying $v_Q$ to
both sides gives
\[
1\leq v_Q(y)+v_Q \left(y^{m-1}-\sum^{m-1}_{i=1}a_iy^{i-1} \right)
=v_Q(a_0)=v_Q(N_{\vF_q(y,T):\vF_q(T)}(y)),
\]
where the norm is equal to $a_0$ up to sign.

Since $N_{\vF_q(y,T):\vF_q(T)}(y)$ belongs to $\vF_q(T)$, its
$v_Q$-valuation is $e$ times its $v_\infty$-valuation and is therefore a
multiple of $e$. Being positive, it is at least $e$.

Observe now that \[N_{\vF_q(y,T):\vF_q(T)}(y)^{[M:\vF_q(T,y)]}=N_{M:\vF_q(T)}(y),\] which also implies that $v_Q(N_{M:\vF_q(T)}(y))/e\geq 1$

Using now \eqref{eq:contrwithval} we get a contradiction by observing that the RHS should be strictly positive, while the LHS is zero.
\end{proof}

\subsection{Drinfeld Modules}
Let $\vF_q$ be a finite field and let $k=\vF_{q^n}$. A \emph{Drinfeld Module} is an injective morphism of $\vF_q$-algebras $\phi: \vF_q[T]\rightarrow k\{\tau\}$. We denote the image of $a\in\vF_q[T]=A$ via $\phi$ by $\phi_a$.
The \emph{rank} of the Drinfeld module is $r=\deg_\tau (\phi_T)$.
Let $\fp\neq 0$ be the characteristic of $\phi$. In the rest of these notes we identify $A$ with  $\vF_q[\phi_T]$.
We denote by $\phi_a(x)$ the $q$-polynomial in $\vF_{q^n}[x]$ attached to $\phi_a$, this is called the $a$-division polynomial of $\phi$. For example, if $\phi_T=\tau^2+\tau+1$ then 
$\phi_T(x)=x^{q^2}+x^q+x$.
Attached to a Drinfeld module $\phi:\vF_q[T]\longrightarrow k\{\tau\}$, defined by $\phi_T=\sum^r_{i=0} t_i\tau^i$ there is a morphism called the \emph{structure morphism} of $\phi$ and defined by
\begin{align*}
\gamma: \vF_q[T] &\longrightarrow \vF_{q^d}\\
a &\mapsto \partial_x \phi_a(x)
\end{align*}
where $\partial_x$ is the usual $x$-derivative of polynomials, and $d=[\vF_q(t_0):\vF_q]$.

The characteristic of $\phi$ is the minimal polynomial $\mathfrak{p}$ over $\vF_q$ of the degree zero term of $\phi_T$, i.e. the minimal polynomial of $\partial_x \phi_T(x)$. 
Given a Drinfeld module $\phi$ of rank $r$, there is an action $*$ of $k[T]$ on the ring $k\{\tau\}$  defined by the following rules:
 if $u\in k\{\tau\}$ and $y\in k$, then $y*u=yu$ and $T*u=u\phi_T$.

Notice that $k$ does not commute with $\phi_T$, but $T$ commutes with $k$ in $k[T]$. This happens simply because $T$ is acting from the right, while $k$ is acting from the left. But this still gives a well defined action because $\vF_q[\phi_T]$ is commutative.
Interestingly enough, with this construction commutativity simply follows by associativity of product.

The following results are classical, and can be found for example in \cite[Chapter 3]{Papikian2023}.
\begin{theorem}[Anderson Motive Decomposition of $k\{\tau\}$ via $\phi$]
The ring $k\{\tau\}$ is a free $k[T]$ module of rank $r$ with free basis $1,\tau,\tau^2,\dots,\tau^{r-1}$.
\end{theorem}

For a $q$-polynomial $f$, let $h(f)$ denote the smallest exponent of $\tau$ appearing in $f$ with nonzero coefficient.
\begin{proposition}
Let $\phi$ be a Drinfeld module having characteristic $\fp$.
There is a unique integer $H(\phi)$ such that 
\[h(\phi_a)=H(\phi)v_{\fp}(a)\deg(\fp).\]
\end{proposition}

Let $\overline k$ be the algebraic closure of $k$. For a Drinfeld module $\phi$ and a polynomial $a\in A$, we denote by $\phi[a]$ the set of roots of $\phi_a(x)$ in $\overline k$. The set $\phi[a]$ is clearly an $\vF_q$-vector space and is called \emph{the $a$-torsion of $\phi$}. The splitting field of $\phi_a(x)$, denoted by $k(\phi[a])$ is called the \emph{$a$-division field of $\phi$}.

\begin{proposition}
Let $\phi$ be a Drinfeld module, $e\in \vN$, and $\fl$ be a prime of $A$. Let  $a,b$ be coprime elements of $A$. 
Then, as $A$-modules, \begin{itemize}
\item $\phi[ab]=\phi[a]\oplus \phi[b]$.
\item if $ \fl \neq \fp$, then $\phi[\fl^e]\cong (A/(\fl^e))^r$
\item $\phi[\fp^e]\cong (A/(\fp^e))^{r-H(\phi)}$
\end{itemize} 

\end{proposition}

The \emph{Tate module} of $\phi$ is defined as 
\[T_{\fl}(\phi)=\lim_{\leftarrow_n} \phi[\fl^n]\]
with transition maps defined by powers of $\phi_{\fl}$.

\begin{theorem}
Let $\fl$ be a prime different from the characteristic of $\phi$. Then, $T_{\fl}(\phi)$ is isomorphic to $(A_\fl)^r$, where $A_\fl$ is the completion of $A$ at $\fl$.
If $\fp$ is the characteristic of $\phi$, then $T_{\fp}(\phi)\cong (A_\fp)^{r-H(\phi)}$
\end{theorem}

\begin{definition}
Let $\phi:\vF_q[T]\longrightarrow k\{\tau\}$ be a Drinfeld module. We say that an element $u\in k\{\tau\}$ is an \emph{endomorphism} for $\phi$ if $\phi_T u=u\phi_T$. 
\end{definition}
Endomorphisms are simply elements of $k\{\tau\}$ that commute with every element of the image of $\phi$.

\begin{proposition}
The ring $\End_k(\phi)$ is a finitely generated $A$-module.
\end{proposition}
\begin{proof}
The ring $k\{\tau\}$ is free of finite rank over $k[T]$, as it decomposes as an Anderson motive as  $k\{\tau\}=\bigoplus^{r-1}_{i=0} k[T]\tau^i$. Therefore, it is also free of finite rank as an $A$-module. Since $A$ is a PID, every submodule of a free module of finite rank is also free of finite rank. Since $\End_k(\phi)$ is an $A$-submodule of $k\{\tau\}$ we have the wanted claim.
\end{proof}

The following is a standard argument, we include the proof for completeness.

\begin{theorem}\label{integrality}
Let $u\in \End_k(\phi)$. Then $u$ is integral over $A$.
\end{theorem}
\begin{proof}
Since $\End_k(\phi)$ is finitely generated over the Noetherian ring $A$, it is Noetherian as an $A$-module. Now, consider the sequence of $A$-submodules 
$M_i=\langle 1,u,\dots u^i\rangle_{A}$. This chain stabilizes at some index $j$, giving $M_j=M_{j-1}$, which forces $u^j\in \langle 1,u,\dots u^{j-1}\rangle_{A}$. This gives a monic relation for $u$ over $A$, proving integrality of $u$.
\end{proof}

\begin{remark}
Endomorphisms of a Drinfeld Module preserve torsion points, i.e. $u(\phi[a])\subseteq \phi[a]$. In fact, if $z\in \phi[a]$ and $u\in \End(\phi)$, then $u(z)$ is annihilated by $\phi_a$:
\[\phi_a(u(z))=u(\phi_a(z))=u(0)=0\]
In particular, an endomorphism of $\phi$ can be seen as an endomorphism of $T_{\fp}(\phi)$, and therefore as an $r\times r$ matrix in $A_{\fp}$.
\end{remark}

The following theorem records the standard relation between the minimal and characteristic polynomials of an endomorphism. A proof is given in Appendix~\ref{app:ell}, Theorem~\ref{app:thm:main}.
Recall that $A=\vF_q[T]$.
\begin{theorem}\label{thm:charpol}
Let $\vF_q$ be a finite field, $n$ be a positive integer, and $k=\vF_{q^n}$. Let $\phi:A\longrightarrow k\{\tau\}$ be a Drinfeld module of rank $r$.
Let $\fl\in A$ be an irreducible polynomial different from the characteristic of $\phi$.
Let $u\in\End(\phi)$. Set $F=\vF_q(T)$ and let $m(T,x)$ be the minimal polynomial of $u$ over $F$; its coefficients lie in $A$ by Theorems~\ref{thm:integrality} and~\ref{integrality}. Let $P_{\phi,\fl,u}(T,x)$ be the characteristic polynomial of $u$ over $T_{\fl}(\phi)$. We have that
$m(T,x)$ is equal to the minimal polynomial of $u$ as an $A_{\fl}$-linear map on $T_\fl(\phi)$. Moreover,
\[P_{\phi,\fl,u}(T,x)=m(T,x)^{r/[F(u):F]}.\]
In particular, $P_{\phi,\fl,u}(T,x)$ does not depend on the choice of $\fl$; we denote this common polynomial by $P_{\phi,u}(T,x)$.
\end{theorem}

\section{The technical core}\label{sec:techcore}
This is the crucial part of the paper that allows to cut out most of the deep results that are needed in other literature such as \cite{anderson1986t,angles1997some,Papikian2023}.

\begin{proposition}\label{thm:switch}
Let $k$ be a finite field and $\phi:\vF_q[T]\longrightarrow k\{\tau\}$ be a Drinfeld module. Let $u\in \End(\phi)$, and  $m_u=m_u(T,x)$ be the minimal polynomial for $u$ over $\vF_q(T)$. Let $\psi:\vF_q[x]\longrightarrow k\{\tau\}$ be the Drinfeld module defined by $\psi_{x}=u$. Then there exists $c\in \vF_q$ such that $c m_u(x,T)\in \vF_q[x][T]$ is the minimal polynomial (in $T$) for $\phi_T$ over $\vF_q(x)$ (where now $x=\psi_x=u$) and has coefficients in $\vF_q[x]$.
\end{proposition}
\begin{proof}
Consider the map
\[E:\vF_q[x,T]\longrightarrow k\{\tau\}\]
\[f(x,T)\mapsto f(u,\phi_T)\]
The kernel of this map is a certain ideal $J$ containing the minimal polynomial $m_u(T,x)$ of $u$ over $\vF_q[T]$. To see the divisibility, view $\vF_q[x,T]$ as $\vF_q[T][x]$ and pass to $\vF_q(T)[x]$, where the kernel is generated by the minimal polynomial of $u$. Since every polynomial in $J$ must be divisible by $m_u(x,T)$ in $\vF_q[T][x]$, this shows that $J=(m_u(x,T))$. Reversing the role of $\phi_T$ and $u$ gives the wanted claim.
\end{proof}

\begin{proposition}\label{thm:reduction_charpol}

Let $k$ be a finite field and let $u$ be an endomorphism of a Drinfeld module $\phi:\vF_q[T]\longrightarrow k\{\tau\}$. Let $\fl$ be a monic irreducible polynomial in $\vF_q[T]$ different from the characteristic of $\phi$.
Let $P_{\phi,\fl,u}(T,x)\in\vF_q[T][x]$ be the characteristic polynomial of $u$ when acting on $T_\fl(\phi)$. Then, $P_{\phi,\fl,u}(T,x)\bmod\fl$ is the characteristic polynomial of $u_{|\phi[\fl]}:\phi[\fl]\longrightarrow\phi[\fl]$.
\end{proposition}
\begin{proof}
Simply observe that $P=\det(x-u_{|A_{\fl}^r})$, then 
\[P\mod \fl=\det(x-u_{|A_{\fl}^r} \mod \fl)=\det(x-u_{|(A/\fl)^r} )=\det(x-u_{|\phi[\fl]} ),\]
which concludes the proof.
\end{proof}

Let $R$ be a PID.
We use the notation $\chi_R$ from Definition \ref{def:orderideal}. With abuse of notation, for an ideal $(\fl^e)$ in $A_{\fl}$, we denote by $\deg((\fl^e))$ the quantity $e\deg(\fl)$.
Notice that if $M$ is finite and $R=A_{\fl}$, then $\dim_{\vF_q}(M)=\deg(\chi_R(M))$.

\begin{proposition}\label{prop:l-primary}
Let $k$ be a finite extension of $\vF_q$.
Let $\phi:\vF_q[T]\longrightarrow k\{\tau\}$ be a Drinfeld module. Let $\fl$ be a prime in $\vF_q[T]$. Let $0\neq u\in \End(\phi)$.
Let $e$ be the minimal integer such that $U_\fl=\phi[\fl^e]\cap \ker(u_{|\overline k})$ is maximal. Then,
\[
v_\fl\big(\det(u\mid T_\fl(\phi))\big)
=\dim_{\vF_q}(U_\fl)/\deg(\fl).
\]
\end{proposition}
\begin{proof}
Since $u\neq0$, its kernel on $\overline k$ is a finite set. It follows that $u$
is injective on $T_\fl(\phi)$: to see this, simply observe that $\bigcup_{n\in\vN} \phi[\fp^n]$ is infinite and $u$ cannot map infinitely many elements of $\overline k$ to zero because of the finiteness of its kernel. By the elementary
divisor theorem over the PID $A_\fl$, there are positive integers $e_i$
such that
\[
T_\fl(\phi)/uT_\fl(\phi)\cong\bigoplus_i A_\fl/(\fl^{e_i}).
\]
In particular this quotient is finite. If it is zero, the claim is
immediate. Otherwise, for every $d\geq\max_i e_i$, the identification
$\phi[\fl^d]\cong T_\fl(\phi)/\fl^dT_\fl(\phi)$ (which is compatible with the
action of $u$), gives
\[
\phi[\fl^d]/u\phi[\fl^d]
\cong T_\fl(\phi)/(uT_\fl(\phi)+\fl^dT_\fl(\phi))
=T_\fl(\phi)/uT_\fl(\phi).
\]
On the other hand, rank--nullity applied to the $\vF_q$-linear map
$u\colon\phi[\fl^d]\to\phi[\fl^d]$ gives
\[
\dim_{\vF_q}\big(\phi[\fl^d]/u\phi[\fl^d]\big)
=\dim_{\vF_q}\big(\phi[\fl^d]\cap\ker(u_{|\overline k})\big).
\]
Since we took $d$ large enough so that the right-hand side has stabilized at
$U_\fl$, and using Theorem~\ref{thm:chidet}, we obtain
\begin{align*}
v_\fl\big(\det(u\mid T_\fl(\phi))\big)
&=v_\fl\big(\chi_{A_\fl}(T_\fl(\phi)/uT_\fl(\phi))\big)\\
&=\frac{\dim_{\vF_q}(T_\fl(\phi)/uT_\fl(\phi))}{\deg(\fl)}\\
&=\frac{\dim_{\vF_q}(U_\fl)}{\deg(\fl)},
\end{align*}
which proves the claim.
\end{proof}

We will now use Proposition \ref{prop:l-primary} to scan the primary parts of the kernel of an endomorphism.
Recall that an element in $k\{\tau\}$ is said to be \emph{separable} if its degree zero coefficient is different from zero.

\begin{corollary}\label{cor:valuationsep}
Let $k$ be a finite extension of $\vF_q$.
Let $\phi:\vF_q[T]\longrightarrow k\{\tau\}$ be a Drinfeld module.
Let $0\neq u\in \End(\phi)$. Then $\deg(P_{\phi,u}(T,0))=\deg(\det(u))=\deg_{\tau}(u)$.
\end{corollary}

\begin{proof}

Let us restrict first to $u$ separable and having $\det(u)$ not divisible by the characteristic $\fp$ of $\phi$.
Write the characteristic polynomial of $u$ as
\[
P_{\phi,u}(T,x)=x^r+\cdots+c_1(T)x+c_0(T),
\qquad c_0(T)=(-1)^r\det(u).
\]
By Cayley--Hamilton on an auxiliary Tate module and the faithfulness of the
Tate-module action, $P_{\phi,u}(T,u)=0$ in $\End(\phi)$. Since $c_0(T)$ is invertible
modulo $\fp$, this identity shows that $u$ is invertible on $\phi[\fp]$.
Set $\phi[\fp^\infty]:=\bigcup_{j\geq1}\phi[\fp^j]$. Consequently
$\ker(u)\cap\phi[\fp^\infty]=0$: otherwise, if $j$ were
minimal such that some nonzero $z\in\ker(u)$ lies in $\phi[\fp^j]$, then
$\phi_{\fp^{j-1}}(z)$ would be a nonzero element of
$\ker(u)\cap\phi[\fp]$.
Let the $U_\fl$ be as in Proposition \ref{prop:l-primary}.
First, we want to prove that $\bigoplus_{\fl} U_\fl$ is a decomposition of the kernel of $u$ on $\overline k$. For every element $y\in \ker(u)$ we have that there is a minimal polynomial $g_y\in \vF_q[T]$ such that $\phi_{g_y}(y)=0$ (this simply follows by the fact that the iterations of $\phi_T$ on $y$ orbit in a finite set). Therefore, there is a $g\in \vF_q[T]$ such that $\phi_g(\ker(u))=0$ (for example, the product of all the $g_y$ work). It follows that we can take $g$ of minimal degree such that $\phi_g(\ker(u))=0$. Suppose that $g=\prod^k_{i=1}\fl_i^{e_i}$, for some primes $\fl_i$ and some positive integers $e_i$ and $k$. By construction, we have that 

\[\ker(u)\subseteq \phi[g]=\bigoplus^k_{i=1} \phi[\fl_i^{e_i}].\]

The projections associated with this primary decomposition are given by
elements of $A$ and therefore commute with $u$. Hence each primary component
of an element of $\ker(u)$ still lies in $\ker(u)$, and it follows that
\[
\ker(u)\subseteq\bigoplus_{i=1}^k
\bigl(\phi[\fl_i^{e_i}]\cap\ker(u)\bigr)
\subseteq\bigoplus_{i=1}^k U_{\fl_i}.
\]

The other inclusion is obvious since $U_\fl\subseteq \ker(u)$.

So we have $\bigoplus^k_{i=1} U_{\fl_i}=\ker(u)$. From this it follows that
\[\dim_{\vF_q}(\bigoplus^k_{i=1} U_{\fl_i})=\deg_{\tau}(u)=:d_u.\] 
Since $\fp\nmid\det(u)$, Theorem~\ref{thm:charpol} and
Proposition~\ref{prop:l-primary} give
\begin{align*}
\deg_T\det(u)
&=\sum_{\fl\neq\fp}\deg(\fl)v_\fl(\det(u))\\
&=\sum_{\fl\neq\fp}\deg(\fl)
v_\fl\bigl(\det(u\mid T_\fl(\phi))\bigr)\\
&=\sum_{\fl\neq\fp}\dim_{\vF_q}(U_\fl)
=d_u,
\end{align*}
as wanted for a separable $u$ whose determinant is not divisible by $\fp$.

Let now $u$ be any endomorphism. 
Let $N$ be a positive integer and consider $w=\phi_{\fp^N}u+1$. The endomorphism $w$ is separable because its constant $\tau$-coefficient is $1$. Moreover, on an auxiliary Tate module $T_\fl(\phi)$ with $\fl\neq\fp$, if $U$ denotes the matrix of $u$, then $w$ has matrix $I_r+\fp^N U$. Hence
\[
\det(w)=\det(I_r+\fp^N U)
=1+\sum_{i=1}^r b_i(T)\fp^{Ni}
\equiv1\pmod{\fp},
\]
where Theorem~\ref{thm:charpol} identifies the first determinant with the
global determinant and ensures that the coefficients $b_i(T)$ belong to
$A$. Thus $\det(w)$ is not divisible by $\fp$. Let $r$ be the rank of $\phi$, so we have
\begin{equation}\label{eq:degdetpert}
\deg_\tau(u)+rN\deg(\fp)
=\deg_\tau(w)
=\deg_T\det(w)
=\deg_T\det(\phi_{\fp^N}u+1),
\end{equation}
where the second equality follows from the case just proved, since $w$ is
separable and $\fp\nmid\det(w)$.
Write
\[
P_{\phi,u}(T,x)=\det(xI_r-U)=\sum_{i=0}^r c_i(T)x^i,
\qquad c_0(T)=(-1)^r\det(u).
\]
Then
\[
\det(I_r+\fp^N U)
=(-1)^rP_{\phi,\phi_{\fp^N}u}(T,-1)
=\sum_{i=0}^r(-1)^{r+i}c_i(T)\fp^{N(r-i)}.
\]
Since $c_0(T)\neq0$, for $N$ sufficiently large the term
$(-1)^rc_0(T)\fp^{Nr}=\det(u)\fp^{Nr}$ has strictly larger $T$-degree
than every term with $i>0$. Therefore
\[
\deg_T\det(\phi_{\fp^N}u+1)
=rN\deg(\fp)+\deg_T\det(u).
\]

Thus combining \eqref{eq:degdetpert} with the equation above we get
\[\deg_\tau(u)+rN\deg(\fp)=rN\deg(\fp)+\deg_T\det(u),\]
which gives the wanted claim for every nonzero $u$.
\end{proof}

We are now ready to prove that $-\deg\det$ induces a valuation on $\vF_q(u,T)$.

\begin{theorem}\label{thm:degdetval}
Let $k$ be a finite extension of $\vF_q$.
Let $\phi:\vF_q[T]\longrightarrow k\{\tau\}$ be a Drinfeld module. Let $v\in \End(\phi)$. Then, $\deg\det(v)=\deg_{\tau}(v)$, and therefore $-\deg\det(\cdot)$ induces a valuation on $R=\vF_q[v,T]$ which can be extended to a valuation of $\vF_q(v,T)$.
\end{theorem}
\begin{proof}
Since $-\deg_\tau(\cdot)$ is a valuation on $R$ (i.e. it verifies the axioms of valuations on the subring $R$), and $\vF_q(v,T)$ is the fraction field of $R$, showing that $\deg\det(\cdot)=\deg_{\tau}(\cdot)$ on $R$ is enough. But this is an immediate consequence of Corollary \ref{cor:valuationsep}.
%
\end{proof}

Let $M$ be a  finite extension of $K$. For any $y\in M$  the \emph{norm} $N_{M:K}$ of $y$ is defined as the determinant of the multiplication by $y$ in $M$, where the multiplication by $y$ is considered as a $K$-linear map. If $F\subseteq K\subseteq M$, recall that the norm satisfies
\[N_{M:F}=N_{K:F}\circ N_{M:K}.\]

Set now $K=\vF_q(T)$ and $N=N_{M:K}$. The  pseudo-absolute value of $y$ is defined as $|y|=q^{-v_{\infty}(N(y))}$, where $v_\infty$ is the usual valuation at $\infty$ of $\vF_q(T)$. Notice that if $y$ is integral over $\vF_q[T]$, then $|y|=q^{\deg(N(y))}$. Notice that in principle, $|\cdot|$ might not be an absolute value.

\begin{lemma}\label{lemma:uniqueplaceourcase}
Let $\phi:\vF_q[T]\rightarrow k\{\tau\}$ be a Drinfeld module and $u$ be an endomorphism of $\phi$. Set $L=\vF_q(T,u)$. The map $y\mapsto v_{\infty}(N_{L:\vF_q(T)}(y))$ is a valuation.
In particular, there is only one place $\tilde\infty$ of $L$ lying above $\infty$.
\end{lemma}
\begin{proof}
Set $F=\vF_q(T)$.
We want to use Lemma \ref{lemma:techsplit} on the extension $L:F$, so we need to show that $v(\cdot)=-\log_q|\cdot|$ is a valuation, where $|\cdot|$ is the pseudo-absolute value defined earlier. It is enough to show that $v$ is a valuation on $R=\vF_q[u,T]$ because $L$ is the fraction field of $R$.
We know that $-\log_q|\cdot|=v_{\infty}(N_{L:F}(\cdot))$. Let us denote by $m^F_y$ the minimal polynomial of $y$ over $F$.
Therefore, for any $y\in R$ we have that
\begin{align*}
v_{\infty}(N_{L:F}(y))=& v_{\infty}(N_{F(y):F}(N_{L:F(y)}(y)))\\
					  =& v_{\infty}(N_{F(y):F}(y)^{[L:F(y)]})\\
					  =& -[L:F(y)]\deg_T(N_{F(y):F}(y))\\
					  =& -[L:F(y)]\frac{[F(y):F]}{r}\deg_T(N_{F(y):F}(y)^{r/[F(y):F]})\\
					  =& -\frac{[L:F]}{r}\deg_T(m^F_y(0)^{r/[F(y):F]})\\
					  =& -\frac{[L:F]}{r}\deg_T(\det(y))\\
					  =& -\frac{[L:F]}{r}\deg_{\tau}(y),
\end{align*}
where the first equality uses norm transitivity, the sixth equality uses Theorem \ref{thm:charpol}, and the last equality uses Theorem \ref{thm:degdetval}.

Thanks to Theorem \ref{thm:degdetval} we know that $-\deg_\tau=-\deg\det$ induces a valuation on $\vF_q(u,\phi_T)$, and so does also $v$, as it is simply a (positive) constant (i.e. 
$[L:K]/r$) times a valuation on $R$ (i.e. $-\deg_\tau(\cdot)$). Lemma \ref{lemma:techsplit} then shows that there is only one place above $\infty$.
\end{proof}

\begin{proposition}\label{prop:technical}
Let $R=\vF_q[T,u]$ and let $\tilde \infty$ be the place lying above $\infty\subset \vF_q(T)$, which we suppose having ramification index equal to $e$. The map $-\deg\det:R\longrightarrow \vZ$ induces a valuation on the fraction field $L$ of  $R$.
Moreover, the valuations given by $v_{\tilde \infty}$ and $v_\tau(r_1/r_2)=\deg_\tau(r_2)-\deg_\tau(r_1)$ are both equivalent to the one induced by $-\deg\det$. In particular, on $\vF_q[u,T]$ they verify the following
\[v_\tau(\cdot):=-\deg_\tau(\cdot)=-\deg(\det(\cdot))=(r/e)v_{\tilde \infty} (\cdot).\]
In particular, if $u=\pi=\tau^n$, where $[k:\vF_q]=n$, then
$v_{\tilde \infty}(\pi)=-en/r$.
\end{proposition}
\begin{proof}
Since there is only one valuation at infinity of $\vF_q(u,T)$ we necessarily have that $-\deg_\tau(\cdot)=c'v_{\tilde\infty}(\cdot)$. To find the proportionality factor we apply the maps to $T$, getting $-r=- c'e$.
\end{proof}

\section{The proof of the Riemann Hypothesis for Drinfeld Modules}\label{sec:proofRH}

Finally, using the technical core, we are ready to conclude the proof in a standard fashion using symmetric polynomials.

\begin{proof}[Proof of Theorem \ref{thm:RH}]
Set $F=\vF_q(T)$.
First, we need to show that $|\cdot|_*$ is an absolute value. It is enough to show that $v:y\mapsto v_{\infty}(N_{F(y):F}(y))/[F(y):F]$ is a valuation on $R=\vF_q[T,u]$. Let $y\in R$. Now simply observe that 
\[
v_{\infty}(N_{F(y):F}(y)^{[L:F(y)]})/([L:F(y)][F(y):F])=v_{\infty}(N_{L:F}(y))/[L:F]\]
Since Lemma \ref{lemma:uniqueplaceourcase}, applied to $u=\pi$, shows that $v_{\infty}(N_{L:F}(y))=-\deg(N_{L:F}(y))$ is a valuation, we have the wanted claim.

Let us now prove (3).
Suppose that $a_0$  is divisible by an irreducible polynomial $\fq\neq \fp$. Let us consider the map given by $\pi:T_\fq(\phi)\rightarrow T_\fq(\phi)$. The determinant of such $A_{\fq}$-linear map is, up to sign, equal to  $P_{\phi,\pi}(T,0)=a_0$ (independently of $\fq$, because the characteristic polynomial of an endomorphism on the $\fq$-adic Tate module is independent of $\fq$, if $\fq$ is chosen different from the characteristic, see Theorem \ref{thm:charpol}).
Using now Proposition \ref{thm:reduction_charpol} we obtain that 
\[\det(\pi_{|\phi[\fq]})= \pm P_{\phi,\pi}(T,0)\mod \fq=0.\]
But this is impossible because $\pi=\tau^n$ is a bijection on $\overline k$, and therefore on $\phi[\fq]$. This shows that the only prime dividing $a_0$ is $\fp$.
Notice that now it is simply enough to prove that $\deg_T(P_{\phi,\pi}(T,0))=n$, but this is immediate from Theorem \ref{thm:degdetval}:  $\deg_\tau(\pi)=n=\deg(\det(\pi))=\deg_T(P_{\phi,\pi}(T,0))$.

Let us prove (1).  We are asking if
\[|\alpha|_*=q^{-v_{\infty}(N(\alpha))/[\vF_q(\alpha,T):\vF_q(T)]}=q^{n/r}.\]
Since $N(\alpha)=N(\pi)$ and $[\vF_q(\pi,T):\vF_q(T)]=[\vF_q(\alpha,T):\vF_q(T)]$ (because, by Theorem \ref{thm:charpol}, $m_{\pi}(T,x)$ is irreducible, and $\alpha$ is also a root of $m_{\pi}(T,x)$), it is enough to show that 
\[
-v_{\infty}(N(\pi))/[\vF_q(\pi,T):\vF_q(T)]=n/r,
\]
equivalently that
\[
\deg(N(\pi)^{r/[\vF_q(\pi,T):\vF_q(T)]})=n,
\]
which holds because
$N(\pi)^{r/[\vF_q(\pi,T):\vF_q(T)]}=\det(\pi)=c\fp^{n/d}$
for some $c\in \vF_q$.

The item (2) is a relatively immediate consequence of (1) by observing that $a_i$, up to sign, is the elementary symmetric polynomial of degree $r-i$ evaluated at the roots of the characteristic polynomial (with repetition).
Let $\overline \infty$ be a place of $M$ lying above $\infty\subseteq \vF_q(T)$. 
For an element of $F=\vF_q(T)$ we have
\[
\deg_T(\cdot)=-v_\infty(\cdot)
=-\frac{v_{\overline\infty}(\cdot)}{e(\overline\infty\mid\infty)}.
\]

For each $t\in \{1,\dots r\}$, choose an $F$-embedding
$\sigma_t:F(\pi)\hookrightarrow M$ such that $\sigma_t(\pi)=\alpha_t$.
The valuation $v_{\overline\infty}\circ\sigma_t$ restricts to a place of
$F(\pi)$ above $\infty$, and hence to its unique such place
$\widetilde\infty$ by Lemma~\ref{lemma:uniqueplaceourcase}. Ramification
indices in the tower therefore give
\[
\frac{v_{\overline\infty}(\alpha_t)}
{e(\overline\infty\mid\infty)}
=\frac{v_{\widetilde\infty}(\pi)}
{e(\widetilde\infty\mid\infty)}
=-\frac nr,
\]
where the last equality follows from Proposition~\ref{prop:technical}.
If $s_{r-i}(\alpha_1,\dots,\alpha_r)=0$, the desired degree bound is
immediate. Otherwise,
\begin{align*}
\deg_T(s_{r-i}(\alpha_1,\dots,\alpha_r))
&=-\frac{v_{\overline\infty}(s_{r-i}(\alpha_1,\dots,\alpha_r))}
{e(\overline\infty\mid\infty)}\\
&\leq-\min_{\substack{S\subseteq\{1,\dots,r\}\\|S|=r-i}}
\sum_{t\in S}\frac{v_{\overline\infty}(\alpha_t)}
{e(\overline\infty\mid\infty)}\\
&=(r-i)n/r,
\end{align*}

which is the claim.

\end{proof}

\section{Acknowledgements}
We thank Mihran Papikian for very interesting discussions that greatly improved the proofs in this manuscript.
This work was supported by NSF grant number 2338424.

\appendix

\section{Independence of $\fl$ of the characteristic polynomial}\label{app:ell}

\begin{center}
	{\large\textit{Appendix by Mihran Papikian}}
\end{center}

In this appendix we give a proof, in the spirit of the main text,
of the statement recorded in Theorem~\ref{thm:charpol}. The argument uses
Proposition~\ref{prop:l-primary}, the product formula, finiteness of ray
class groups for the function field \(F(u)\), and the standard
Laurent-series embedding of endomorphism algebras, recalled from
\cite[Corollary 3.4.11]{Papikian2023}.

\subsection{Setup}
Let $A=\vF_q[T]$, $F=\vF_q(T)$, and let $\infty$ be the place of $F$ with
$v_\infty(a)=-\deg(a)$, so $\deg(\infty)=1$. Let $k=\vF_{q^n}$ and let
$\phi\colon A\to k\{\tau\}$ be a Drinfeld module of rank $r$ with
characteristic $\fp$, whose monic generator we denote $\wp\in A$. 

Fix $0\neq u\in\End_k(\phi)$ and set \[E=F(u)\subseteq \End_k(\phi)\otimes_A F.\] By
Theorem~\ref{integrality} and Theorem~\ref{thm:integrality}, the minimal
polynomial of $u$ over $F$,  
\[
m(x)\in A[x], 
\]
is monic with coefficients in $A$. Put \[\delta=[E:F]=\deg_x m(x);\]
we write $\delta$ rather than $d$ to avoid a conflict with $d=\deg(\fp)$
from the main text. For a prime
$\fl\neq\fp$ of $A$, let $T_\fl=T_\fl(\phi)$ be the Tate module,
$V_\fl=T_\fl\otimes_{A_\fl}F_\fl$, and
\[
P_{\fl}(x)=\det\big(x-u\mid T_\fl\big)\in A_\fl[x],
\]
which abbreviates the characteristic polynomial $P_{\phi,\fl,u}(T,x)$ of
Theorem~\ref{thm:charpol}. Finally, for $\beta\in k\{\tau\}$ we denote by
$\partial\beta$ the constant $\tau$-coefficient of $\beta$, so that
$\partial\phi_a=\gamma(a)$ for $a\in A$.

The main theorem of the appendix is the following:

\begin{theorem}\label{app:thm:main}
For every prime $\fl\neq\fp$,
\[
P_{\fl}(x)=m(x)^{r/\delta},
\]
and $m(x)$ is the minimal polynomial of $u$ acting on $T_\fl$. In particular
$\delta\mid r$, and $P_{\fl}$ has coefficients in $A$ which are independent of
$\fl$.
\end{theorem}

\subsection{The local decomposition}
The completion $F_\fl$ is a separable extension of $F$ (see
\cite[Theorem 2.8.5]{Papikian2023}). Consequently $E\otimes_F F_\fl$ is
\emph{reduced} for every finite extension $E/F$, including inseparable ones.
Hence
\begin{equation}\label{app:eq:tensor}
E\otimes_F F_\fl\;\cong\;\prod_{w\mid\fl}E_w,
\qquad
m(x)=\prod_{w\mid\fl}m_w(x),
\end{equation}
a product of fields indexed by the places $w$ of $E$ over $\fl$, with the
$m_w(x)\in F_\fl[x]$ monic irreducible and \textit{pairwise distinct}, and
$E_w\cong F_\fl[x]/(m_w(x))$. We write $v_w$ for the \textit{normalized} valuation
of $E$ at $w$, and $e_w,f_w$ for the ramification index and residue degree, so
that \[[E_w:F_\fl]=e_wf_w, \quad \deg(w)=f_w\deg(\fl), \quad v_w|_{F_\fl}=e_w\,v_\fl.\] For $\beta\in E$ we denote by $\beta_w\in E_w$ its
image under $E\hookrightarrow E\otimes_F F_\fl\twoheadrightarrow E_w$; then
$E_w$ is the completion of $E$ at $w$, $v_w(\beta)=v_w(\beta_w)$, and the
isomorphism above sends $u_w$ to the class of $x$. For $0\neq\beta\in E$,
\begin{equation}\label{app:eq:normval}
v_\fl\big(\Nr_{E_w:F_\fl}(\beta_w)\big)=f_w\,v_w(\beta).
\end{equation}

The algebra $E\otimes_F F_\fl$ acts on $V_\fl$, so $V_\fl$ decomposes along the
idempotents of \eqref{app:eq:tensor}:
\[
V_\fl=\bigoplus_{w\mid\fl}V_w,
\qquad
c_w\colonequals \dim_{E_w}V_w\in\vZ_{\geq0},
\]
with $u$ acting on $V_w\cong E_w^{\oplus c_w}$ as the scalar $u_w$. Since the
characteristic polynomial of $u_w$ on $E_w$ over $F_\fl$ is $m_w(x)$ (companion
matrix of the generator), taking blocks gives
\begin{equation}\label{app:eq:Pfactor}
P_{\fl}(x)=\prod_{w\mid\fl}m_w(x)^{c_w},
\qquad
r=\sum_{w\mid\fl}c_w\,[E_w:F_\fl].
\end{equation}
Theorem~\ref{app:thm:main} is thus the assertion that $c_w=r/\delta$ for every $w$.

\subsection{The place at infinity, and the key identity} Let $\cO_E$ be the integral closure of $A$ in $E$, and let 
\(
\cO\colonequals E\cap\End_k(\phi). 
\)
Note that $\cO$ is an order in $\cO_E$: it is a finitely generated
$A$-module, being a submodule of $\End_k(\phi)$; it contains $A[u]$, which
has rank $\delta$ over $A$; and $\cO\subseteq\cO_E$ by
Theorem~\ref{integrality}. Hence $\cO$ is an $A$-subring of $\cO_E$ of full
rank $\delta$.

\begin{proposition}\label{app:prop:infty}
$E$ has exactly one place $\tinf$ above $\infty$, and for
$0\neq\beta\in\cO$,
\begin{equation}\label{app:eq:degtau}
-\deg_\tau(\beta)\;=\;\frac{r}{e_\infty}\,v_{\tinf}(\beta),
\qquad e_\infty \colonequals e(\tinf\mid\infty).
\end{equation}
Moreover, $\deg(\tinf)=f_\infty$ and $e_\infty f_\infty=\delta$.
\end{proposition}

\begin{proof}
	The ring of skew Laurent series $\ls{k}{\tau^{-1}}$ is a division ring on which
	$-\deg_\tau$ is a valuation, and $\End_k(\phi)$ embeds into it. 
	The embedding \(\End_k(\phi)\hookrightarrow \ls{k}{\tau^{-1}}\) extends, after
	inverting \(\phi(A)\), to an embedding
	\[
	F\otimes_A\End_k(\phi)\hookrightarrow \ls{k}{\tau^{-1}}.
	\]
	Restricting this embedding to \(E=F(u)\) gives a valuation of \(E\) lying over
	\(\infty\), as
	$-\deg_\tau(\phi_a)=r\,v_\infty(a)$ for $a\in A$. By \cite[Corollary 3.4.11]{Papikian2023}, this place is the
	\textit{only} one over $\infty$. 
	
	Granting uniqueness, $-\deg_\tau$ and $v_{\tinf}$ are proportional
	on $E^\times$, and the constant is read off from $\beta=\phi_T$:
	$-\deg_\tau(\phi_T)=-r$, $v_{\tinf}(T)=-e_\infty$. Finally
	$\deg(\tinf)=f_\infty\deg(\infty)=f_\infty$, and by \eqref{app:eq:tensor} at
	$\infty$ (the decomposition holds verbatim there, as $F_\infty/F$ is
	separable as well) together with uniqueness,
	\[e_\infty f_\infty=[E_{\tinf}:F_\infty]=\dim_{F_\infty}(E\otimes_FF_\infty)=\delta.\]
\end{proof}

\begin{remark}
	Note that Proposition \ref{app:prop:infty} gives an independent proof of 
	Lemma~\ref{lemma:uniqueplaceourcase}. This is needed because the proof of 
	Lemma~\ref{lemma:uniqueplaceourcase} in the main text invokes Theorem \ref{thm:charpol} inside the displayed chain, which 
	would be circular here. 
\end{remark}

For $0\neq\beta\in\End_k(\phi)$, define the \textit{$\fp$-defect}
\[
D(\beta)\colonequals h(\beta)+\dim_{\vF_q}\big(\ker\beta(\overline k)\cap\phi[\fp^\infty]\big)
\;\geq\;0 ,
\]
where $h(\beta)$ is the height (smallest $\tau$-exponent) of $\beta$.

\begin{proposition}[Key identity]\label{app:prop:master}
For every $0\neq\beta\in\cO$,
\begin{equation}\label{app:eq:master}
\sum_{w\,\nmid\,\fp\infty}c_w\,\deg(w)\,v_w(\beta)
\;=\;\deg_\tau(\beta)-D(\beta).
\end{equation}
\end{proposition}

\begin{proof}
Writing $\beta=\beta'\tau^{h(\beta)}$ with $\beta'$ separable, we have 
\[\dim_{\vF_q}\ker\beta(\overline k)=\deg_\tau(\beta)-h(\beta).\] The kernel is a
finite $A$-module via $\phi$, so it splits into primary components, and
\[
\deg_\tau(\beta)-h(\beta)
=\sum_{\fl\neq\fp}\dim_{\vF_q}\big(\ker\beta\cap\phi[\fl^\infty]\big)
+\dim_{\vF_q}\big(\ker\beta\cap\phi[\fp^\infty]\big).
\]
By Proposition~\ref{prop:l-primary}, the $\fl$-term equals
$\deg(\fl)\,v_\fl\!\big(\det(\beta\mid T_\fl)\big)$ (the proposition applies
since $\beta\neq0$ in the field $E$, so $\beta$ is injective on $T_\fl$). On the other hand, $\beta$
acts on $V_w$ as the scalar $\beta_w\in E_w$, so, computing the determinant
block by block on $V_\fl=\bigoplus_w V_w$ with $V_w\cong E_w^{\oplus c_w}$, and
using that the determinant of multiplication by $\beta_w$ on the $F_\fl$-vector
space $E_w$ is by definition $\Nr_{E_w:F_\fl}(\beta_w)$, we get 
\[
\det\big(\beta\mid V_\fl\big)=\prod_{w\mid\fl}\Nr_{E_w:F_\fl}(\beta_w)^{c_w}. 
\]
Thus, by \eqref{app:eq:normval},
\[
\deg(\fl)\,v_\fl\!\big(\det(\beta\mid T_\fl)\big)
=\sum_{w\mid\fl}c_w\,\deg(w)\,v_w(\beta).
\]
Summing over $\fl\neq\fp$ gives \eqref{app:eq:master}.
\end{proof}

\begin{corollary}\label{app:cor:key}
Set $s\colonequals r/\delta$. For every $0\neq\beta\in\cO$,
\begin{equation}\label{app:eq:key}
\sum_{w\,\nmid\,\fp\infty}\big(c_w-s\big)\,\deg(w)\,v_w(\beta)
\;=\;s\!\!\sum_{w\mid\fp}\deg(w)\,v_w(\beta)\;-\;D(\beta).
\end{equation}
\end{corollary}

\begin{proof}
By the product formula on $E$ (i.e.\ the fact that principal divisors have
degree zero, see \cite[Corollary 1.4.12]{Stichtenoth}) and
Proposition~\ref{app:prop:infty},
\[
\sum_{w\nmid\infty}\deg(w)\,v_w(\beta)
=-\deg(\tinf)\,v_{\tinf}(\beta)
=f_\infty\cdot\frac{e_\infty}{r}\,\deg_\tau(\beta)
=\frac{\delta}{r}\,\deg_\tau(\beta).
\]
Multiply by $s$ and subtract \eqref{app:eq:master}.
\end{proof}

\begin{lemma}\label{app:lem:pkill}
If $\beta\in\cO$ satisfies $\beta-1\in \wp\,\cO$, then $D(\beta)=0$ and
$v_w(\beta)=0$ for every $w\mid\fp$.
\end{lemma}

\begin{proof}
Write $\beta=1+\phi_\wp\beta'$, $\beta'\in\cO$. Then
$\partial\beta=1+\gamma(\wp)\,\partial\beta'=1\neq0$ (as $\gamma(\wp)=0$), so
$h(\beta)=0$. On $\phi[\fp]$ we have $\beta=1$, since $\beta'$ preserves
$\phi[\fp]$ and $\phi_\wp$ kills it. 

If \(0\neq x\in \ker\beta\cap\phi[\fp^\infty]\) and \(j\) is minimal with
\(\phi_{\wp^j}(x)=0\), then \(y=\phi_{\wp^{j-1}}(x)\) is a nonzero element of
\(\phi[\fp]\); but \(\beta(y)=\phi_{\wp^{j-1}}(\beta(x))=0\), while
\(\beta\) acts as the identity on \(\phi[\fp]\), a contradiction.
Hence $\ker\beta\cap\phi[\fp^\infty]=0$ and
$D(\beta)=0$. 

Finally, $\beta\equiv1$ in $\cO/\wp\cO$ makes $\beta$ a unit of
$\cO\otimes_A A_\fp$ (as $\wp$ lies in the Jacobson radical of this semilocal
ring); its inverse lies in $\cO\otimes_A A_\fp\subseteq\cO_E\otimes_A A_\fp$, so
$\beta$ is a unit at every $w\mid\fp$.
\end{proof}

\subsection{Proof of Theorem~\ref{app:thm:main}}
Let
$\ff\subseteq\cO_E$ be the conductor of the order $\cO$ in $\cO_E$; it
is a nonzero $\cO_E$-ideal contained in $\cO$. We prove that $c_w=s=r/\delta$ for
every finite place $w\nmid\fp$ of $E$, in two stages.

\medskip\noindent
\emph{Stage 1: places prime to the conductor.} Fix the modulus
$\fm\colonequals \wp\,\ff\subseteq \wp\,\cO$. Let $w_0$ be any place of $E$ with
$w_0\nmid\fm\infty$, and let $\fw_0\subset\cO_E$ be the corresponding maximal
ideal. The ray class group of $\cO_E$ with modulus $\fm$ is finite: this
follows from the finiteness of the ideal class group of $\cO_E$ (which in
turn follows from the finiteness of the degree-zero divisor class group of
$E$, see \cite[Proposition 5.1.3]{Stichtenoth}) and the finiteness of
$(\cO_E/\fm)^\times$. Hence there are $\nu\geq1$ and $\beta\in\cO_E$ with
\[
\fw_0^{\,\nu}=(\beta),\qquad \beta\equiv1\Mod{\fm}.
\]
Then $\beta\in1+\fm\subseteq1+\wp\,\cO\subseteq\cO$, so $\beta$ is an
endomorphism. Thus, Corollary~\ref{app:cor:key} applies, and by
Lemma~\ref{app:lem:pkill} the right-hand side of \eqref{app:eq:key} vanishes.
By the factorization $\fw_0^{\,\nu}=(\beta)$ in the Dedekind ring $\cO_E$, we have 
$v_{w_0}(\beta)=\nu>0$ and $v_w(\beta)=0$ for every finite place $w\neq w_0$. So
\eqref{app:eq:key} collapses to $(c_{w_0}-s)\deg(w_0)\,\nu=0$, i.e.\
$c_{w_0}=s=r/\delta$. In particular $r/\delta\in\vZ$, so $\delta\mid r$; this
proves $c_w=s$ for every finite place of $E$ dividing neither $\fp$ nor $\ff$.

\medskip\noindent
\emph{Stage 2: places dividing the conductor.} It remains to treat the
finitely many places $w_0\mid\ff$ with $w_0\nmid\fp$. The Stage~1 construction
cannot reach such a place: every $\cO_E$-ideal contained in $\cO$ is contained
in $\ff$, hence divisible by $\fw_0$, so any $\beta\equiv1$ modulo such an
ideal is a unit at $w_0$. Instead we control valuations at all places over a
fixed prime simultaneously.

Fix a prime $\fl_0\neq\fp$ of $A$ such that some place of $E$ dividing $\ff$
lies over $\fl_0$, and let $S$ be the finite set of primes
$\fl'\notin\{\fl_0,\fp\}$ with the same property. Choose $M\geq1$ with $M
e_w\geq v_w(\ff)$ for all $w\mid\fl_0$, so that
inside $\cO_E\otimes_A A_{\fl_0}=\prod_{w\mid\fl_0}\cO_{E_w}$,
\begin{equation}\label{app:eq:conductor}
\fl_0^{\,M}\big(\cO_E\otimes_A A_{\fl_0}\big)
\subseteq\ff\big(\cO_E\otimes_A A_{\fl_0}\big)
\subseteq\cO\otimes_A A_{\fl_0}.
\end{equation}

Let $(t_w)_{w\mid\fl_0}$ be arbitrary non-negative integers and set
$C:=1+\max_w t_w$. 
By weak approximation \cite[Theorem 1.3.1]{Stichtenoth} choose \(y_0\in E^\times\) with
\(v_w(y_0)=t_w\) for all \(w\mid\fl_0\). Multiplying \(y_0\) by an element of
\(A\) not divisible by \(\fl_0\), and with sufficiently large zeros at the
finite poles of \(y_0\) away from \(\fl_0\), we may assume that
\(y\colonequals ay_0\in\mathcal O_E\), without changing the valuations at the places
\(w\mid\fl_0\).

By
\eqref{app:eq:conductor}, $\fl_0^M y\in\cO\otimes_A A_{\fl_0}$, with
$v_w(\fl_0^M y)=M e_w+t_w$ for $w\mid\fl_0$. Since
$\cO/\fl_0^{M+C}\cO\cong(\cO\otimes_A A_{\fl_0})/\fl_0^{M+C}(\cO\otimes_A
A_{\fl_0})$, the class of $\fl_0^M y$ defines a class in $\cO/\fl_0^{M+C}\cO$,
and by the CRT in $\cO$ (the moduli are pairwise coprime in $A$, and $\cO$ is a
free $A$-module) there is $\beta\in\cO$ with
\[
\beta\equiv\fl_0^M y\Mod{\fl_0^{M+C}\cO},\quad
\beta\equiv1\Mod{\wp\,\cO},\quad
\beta\equiv1\Mod{\fl'\cO}\ (\fl'\in S).
\]
We record the local properties of $\beta$: 
\begin{itemize}
	\item[(a)] at $w\mid\fl_0$, \[v_w(\beta-\fl_0^M y)\geq(M+C)e_w>Me_w+t_w,\] so
	$v_w(\beta)=Me_w+t_w$ exactly (in particular $\beta\neq0$); 
	\item[(b)] at $w\mid\fl'$,
	$\fl'\in S$, $v_w(\beta-1)\geq e_w>0$, so $v_w(\beta)=0$; 
	\item[(c)] $\beta-1\in
	\wp\,\cO$, so by Lemma~\ref{app:lem:pkill} the right-hand side of
	\eqref{app:eq:key} vanishes; 
	\item[(d)] at every other finite $w\nmid\fp$ one has
	$w\nmid\ff$, so $c_w=s$ by Stage~1 and its term vanishes identically. 
\end{itemize}
With
(a)--(d), the identity \eqref{app:eq:key} collapses to
\begin{equation}\label{app:eq:stage2}
\sum_{w\mid\fl_0}\big(c_w-s\big)\,\deg(w)\,\big(M e_w+t_w\big)=0,
\end{equation}
valid for every tuple $(t_w)$. Subtracting \eqref{app:eq:stage2} for $t=0$ from
the same with $t_{w_0}=1$, $t_w=0$ ($w\neq w_0$) gives
$(c_{w_0}-s)\deg(w_0)=0$, i.e.\ $c_{w_0}=s$. Running over $w_0\mid\fl_0$ and
over the finitely many such $\fl_0$ completes Stage~2.

\medskip
Together the two stages give $c_w=s$ for every finite $w\nmid\fp$, so
\eqref{app:eq:Pfactor} gives $P_{\fl}(x)=m(x)^{r/\delta}$ for every $\fl\neq\fp$.
For the minimal polynomial: every irreducible factor of the characteristic
polynomial divides the minimal polynomial $g_\fl$ of $u$ on $V_\fl$, so
$m_w\mid g_\fl$ for all $w\mid\fl$; since the $m_w$ are pairwise distinct,
$m=\prod_w m_w$ divides $g_\fl$, and $g_\fl\mid m$ because $m(u)=0$; hence
$g_\fl=m$. The same holds on $T_\fl$, since $m$ is monic in $A_\fl[x]$, kills
$T_\fl$, and nothing of smaller degree can kill $T_\fl$, which spans $V_\fl$.
\qed

\renewcommand{\refname}{References}
\bibliographystyle{plain}
\bibliography{ref}{}

\end{document}